\newcommand{\comment}[1]{}

\documentclass[a4paper,preprint,authoryear]{elsarticle}
\usepackage[all]{xy} % para incluir diagramas
\usepackage{amsfonts, amssymb, amsmath} % pacote AMS
\usepackage{amsthm}

\newtheorem{Thm}{Theorem}[section]
\newtheorem{Lem}{Lemma}[section]
\newtheorem{Prop}{Proposition}[section]
\newtheorem{Proper}{Property}[section]
\newtheorem{Coro}{Corollary}[section]
\newtheorem{Algo}{Algorithm}

\theoremstyle{definition}
\newtheorem{Rem}{Remark}[section]
\newtheorem{Def}{Definition}[section]
\newtheorem{Assum}{Assumption}
\newtheorem{Ques}{Question}[section]
\newtheorem{Example}{Example}[section]

\numberwithin{equation}{section}
\comment{
\newtheorem{Thm}{Theorem}
\newtheorem{Lem}{Lemma}
\newtheorem{Prop}{Proposition}

\newtheorem{Coro}{Corollary}

\theoremstyle{definition}
\newtheorem{Rem}{Remark}
\newtheorem{Def}{Definition}

\newtheorem{Example}{Example}
}

\newcommand{\ind}{{\bf 1}}
\newcommand{\proba}{\,\mathbb P}
\newcommand{\esp}{\,{\mathbb E}}
\newcommand{\supp}{{\rm{supp}}}
\newcommand{\sign}{{\rm{sign}}}

\newcommand{\defe}{\mathrel{\mathop:}=}

\newcommand{\defd}{\eqd}

\newcommand{\calE}{{\cal E}}
\newcommand{\filF}{{\cal F}}

\newcommand{\calN}{{\cal N}}

\newcommand{\calR}{{\cal R}}

\newcommand{\calW}{{\cal W}}

\def\indt#1{\{#1_t\}_{t\in T}}
\def\indtr#1{\{#1_t\}_{t\in\mathbb R}}
\def\indtrn#1{\{#1_t\}_{t\in\mathbb R^n}}

\newcommand{\csspan}{\overline{{\rm{span}}}}

\newcommand{\cmsspan}{\overline{\vee\mbox{-}{\rm{span}}}}
\newcommand{\psspan}{{\rm{span}}_+}
\newcommand{\cpsspan}{\overline{{\rm{span}}_+}}

\def\laps{{L^\alpha_+(S,\mu)}}
\def\las{{L^\alpha(S,\mu)}}

\def\lapdo{{L^\alpha_+(S_1,\mu_1)}}
\def\lapim{{L^\alpha_+(S_2,\mu_2)}}
\def\lap{L^\alpha_+}
\def\la{L^\alpha}

\newcommand{\eqnh}{\begin{eqnarray*}}
\newcommand{\eqne}{\end{eqnarray*}}
\newcommand{\eqnhn}{\begin{eqnarray}}
\newcommand{\eqnen}{\end{eqnarray}}
\newcommand{\equh}{\begin{equation}}
\newcommand{\eque}{\end{equation}}

\newcommand{\sumin}{\sum_{i=1}^n}

\def\bveejp{\bigvee_{j=1}^p}

\newcommand{\eqd}{\stackrel{\rm d}{=}}

\newcommand{\widebar}{\overline}
\newcommand{\eintt}{\ \int^{\!\!\!\!\!\!\!e}}
\newcommand{\Eintt}{\int^{\!\!\!\!\!\!\!e}}

\def\topp#1{^{(#1)}}

\def\lap{L^\alpha_+}

\def\ae{\mbox{-a.e.}}

\def\inftydots#1{#1_1,#1_2,\dots}

\def\mrho{\rho_{\mu,\alpha}}

\def\eratios#1{\calR_{e,+}({#1})}
\def\eratiosf{\calR_{e,+}(\filF)}
\def\nn#1{\left\|#1\right\|}
\def\bnn#1{\Big\|#1\Big\|}

\def\bccbb#1{\Big\{#1\Big\}}

\def\mam{M_{\alpha,\vee}}
\def\mas{M_{\alpha,+}}

\def\d{{\rm d}}

\def\newpara#1{\noindent{\bf#1}\\}

\def\qmand{\quad\mbox{ and }\quad}

\def\mmon{\mbox{ on }}

\def\adaptF#1{\{#1_t,\filF_t:0\leq t<\infty\}}

\def\itemnumber#1{\noindent\parbox{0.23in}{$($#1$)$}}

\journal{Statistics and Probability Letters}
\begin{document}
\begin{frontmatter}

\title{On the Association of Sum-- and Max-- Stable Processes}%\tnoteref{t1}}
%\tnotetext[t1]{The authors were partially supported by NSF grant DMS--0806094 at the University of Michigan.}

\author[um]{Yizao Wang\corref{cor1}}
\ead{yizwang@umich.edu}
\author[um]{Stilian A. Stoev}
\ead{sstoev@umich.edu}
\cortext[cor1]{Corresponding author.}
\address[um]{Department of Statistics, University of Michigan, 439 W.\ Hall, 1085 S.\ University, Ann Arbor,
  MI, 48109--1107}

\begin{abstract}
We address the notion of association of sum-- and max-- stable processes from the perspective of linear and max--linear isometries. 
We establish the appealing results that these two classes of isometries can be identified on a proper space (the extended positive ratio space). This
yields a natural way to associate to any max--stable process a sum--stable process. By using this association, we establish connections between
structural and classification results for sum-- and max-- stable processes.
\end{abstract}

%% keywords here, in the form: keyword \sep keyword

%% PACS codes here, in the form: \PACS code \sep code

%% MSC codes here, in the form: \MSC code \sep code
%% or \MSC[2008] code \sep code (2000 is the default)

\end{frontmatter}

%%%%%%%%%%%%%%%%%%%%%%%%%%%%%%%%%%%%%%%%%%%%%%%%%%%%%%%%%%%%%%%%%%%%%%%%%%%%%%%%%%%%
%%%%%%%%%%%%%%%%%%%%%%%%%%%%%%%%%%%%%%%%%%%%%%%%%%%%%%%%%%%%%%%%%%%%%%%%%%%%%%%%%%%%
%%%%%%%%%%%%%%%%%%%%%%%%%%%%%%%%%%%%%%%%%%%%%%%%%%%%%%%%%%%%%%%%%%%%%%%%%%%%%%%%%%%%
%%%%%%%%%%%%%%%%%%%%%%%%%%%%%%%%%%%%%%%%%%%%%%%%%%%%%%%%%%%%%%%%%%%%%%%%%%%%%%%%%%%%
%%%%%%%%%%%%%%%%%%%%%%%%%%%%%%%%%%%%%%%%%%%%%%%%%%%%%%%%%%%%%%%%%%%%%%%%%%%%%%%%%%%%
\section{Introduction}
Sum--stable processes and max--stable processes are two classes of stochastic processes, which have been investigated for a long time. For sum--stable processes, many results are available about their structure and representations as well as their ergodic properties (see e.g.~\citet{samorodnitsky94stable},~\citet{rosinski95structure},~\citet{rosinski00decomposition} and~\citet{samorodnitsky05null}).
At the same time, the max--stable processes have been relatively less explored from this perspective. However, several recent results imply close connection between the two classes of processes (see e.g.~\citet{dehaan84spectral},~\citet{stoev06extremal},~\citet{stoev08ergodicity} and~\cite{kabluchko08spectral}). 

In this paper, we address the problem of relating these two classes of processes in terms of their spectral representations.
We want to point out that a similar treatment of the association was recently proposed by \cite{kabluchko08spectral}. There 
the author associated the two classes of processes via their spectral measures. His approach and ours, although 
different, lead to the same association. The two approaches together complete the picture of the associations of the 
sum-- and max--stable processes. 
\medskip\\
We start by reviewing the sum-- and max--stable distributions and we will observe strong similarities between the two worlds. The understanding of these similarities is our main motivation for this work. A random variable $X$ is said to have {\it sum--stable} distribution, if for any $a,b\in\mathbb R$, there exists $c>0$ and $d\in\mathbb R$ such that
\[
aX_1 + bX_2 \eqd cX + d\,,
\]
where $X_1$ and $X_2$ are independent copies of $X$. On the other hand, a random variable $Y$ is said to have {\it max--stable} distribution, if for any $a,b>0$, there exists $c>0$ and $d\in\mathbb R$ such that
\[
aY_1 \vee bY_2\equiv \max(aY_1,bY_2)  \eqd cY + d\,,
\]
where $Y_1$ and $Y_2$ are independent copies of $Y$. For simplicity, in this paper we will concentrate on \textit{symmetric $\alpha$--stable} (S$\alpha$S) distributions and \textit{$\alpha$--Fr\'echet} distributions. The S$\alpha$S distribution is a specific sum--stable distribution with characteristic function
\[
\esp\exp\{-itX\} = \exp\{-\sigma^\alpha|t|^\alpha\}\,,\forall t\in\mathbb R\,.
\]
The sum--stability requires that $\alpha\in(0,2]$.
The $\alpha$--Fr\'echet distribution is a specific max--stable distribution such that
\[
\proba(Y\leq y) = \exp\{-\sigma^\alpha y^{-\alpha}\}\,,\forall y\in(0,\infty)\,.
\]
Here $\alpha$ is in $(0,\infty)$. Both $\sigma$'s above are positive and are referred to as the \textit{scale coefficient}.

More similarities can be observed between the \textit{S$\alpha$S} and \textit{$\alpha$--Fr\'echet processes}. An S$\alpha$S process $\indt X$ is a stochastic process, such that any finite linear combination (in form of $\sum_{i=1}^na_iX_{t_i}\,,a_i\in\mathbb R,t_i\in T, n\in\mathbb N$) is S$\alpha$S. Any S$\alpha$S process has integral representation with the form
\equh\label{rep:integralRep}
\indt X \eqd \left\{\int_S f_t(s)d \mas(s)\right\}_{t\in T}\,.
\eque
Here $\indt f\subset\la(S,\mu)$, `$\int$' stands for the \textit{stable integral} and $\mas$ is a {\it S$\alpha$S random measure} on measure space $(S,\mu)$ with control measure $\mu$ (see Chapters 3 and 13 in~\citet{samorodnitsky94stable}). At the same time, an $\alpha$--Fr\'echet process is a stochastic process, such that any finite max--linear combination (in form of $\bigvee_{i=1}^na_iY_{t_i}\,,a_i\geq 0,t_i\in T,n\in\mathbb N$) is $\alpha$--Fr\'echet. Such processes have extremal integral representations of the form
\equh\label{rep:extremalRep}
\indt Y \eqd \left\{\Eintt_S f_t(s)d \mam(s)\right\}_{t\in T}\,.
\eque
Here $\indt f\subset\lap(S,\mu)\defe\{f\in\la(S,\mu):f\geq 0\}$, `$\eintt\ \ $' stands for the \textit{extremal integral} and $\mam$ is an \textit{$\alpha$--Fr\'echet random sup--measure} with control measure $\mu$ (see~\citet{stoev06extremal}). 
The functions $\indt f$ in~\eqref{rep:integralRep} and~\eqref{rep:extremalRep} are called the \textit{spectral functions} of the sum-- or max--stable processes, respectively.
In this paper, $T$ denotes an arbitrary index set, which is sometimes equipped with a measure $\lambda$. Two common settings are $T = \mathbb Z$ with $\lambda$ being the counting measure and $T = \mathbb R$ with $\lambda$ being the Lebesgue measure. Brief summaries of useful properties of stable and extremal integrals are given in Section~\ref{sec:prelim}.

The representation~\eqref{rep:integralRep} implies that
\equh\label{eq:3}
\esp\exp\bccbb{-i\sum_{j=1}^na_jX_{t_j}} = \exp\bccbb{-\int_S\Big|\sum_{j=1}^na_jf_{t_j}(s)\Big|^\alpha d\mu(s)}\,,\quad a_j\in \mathbb R\,,t_j\in T\,,n\in\mathbb N\,,
\eque
which determines the finite--dimensional distributions (f.d.d.) of the S$\alpha$S process $\indt X$. The f.d.d.\ of the $\alpha$--Fr\'echet process $\indt Y$ in~\eqref{rep:extremalRep}, on the other hand, are expressed as:
\equh\label{eq:4}
\proba(Y_{t_1}\leq a_1,\dots,Y_{t_n}\leq a_n) = \exp\bccbb{-\int_S\Big(\bigvee_{j=1}^nf_{t_j}(s)/a_j\Big)^\alpha d\mu(s)}\,,\quad a_j\geq 0\,,t_j\in T\,,n\in\mathbb N\,.
\eque
Note that the r.h.s.\ of~\eqref{eq:3} and~\eqref{eq:4} are similar. Indeed, they both involve exponentials of either linear ($\sum$) or max--linear ($\bigvee$) combinations of spectral functions. The characterizations~\eqref{eq:3} and~\eqref{eq:4} and their close connections play an important role throughout this paper.

Based on these two similar representations, many analogous results have been obtained for sum--stable and max--stable processes, independently.
For example, in the seminal work~\citet{rosinski95structure}, Rosi\'nski established the \textit{conservative--dissipative decomposition} for stationary S$\alpha$S process $\indt X$. This decomposition can be written as
\equh\label{eq:5}
\indt X \eqd\{X^C_t+X^D_t\}_{t\in T}\,.
\eque
Here, we consider $T = \mathbb R$ or $\mathbb Z$, and the two components $\indt {X^C}$ and $\indt {X^D}$ are stochastically independent and arise from 
the \textit{flow} structure induced by the \textit{spectral functions} $\indt f$ of $\indt X$. (As we do not need any specific properties of flows in 
this paper, we refer the readers to~\citet{aaronson97introduction} and~\citet{krengel85ergodic}.)
Recently, an analogous decomposition for max--stable processes has been developed in~\citet{wang09structure}. That is, any measurable 
stationary $\alpha$--Fr\'echet process $\indt Y$, has the decomposition
\equh\label{eq:6}
\indt Y \eqd \{Y^C_t\vee Y^D_t\}_{t\in T}\,,
\eque
where the components $\{Y^C_t\}_{t\in T}$ and $\{Y^D_t\}_{t\in T}$ are independent and also arise from certain types of flows.
It turns out that the corresponding components in the decompositions~\eqref{eq:5} and~\eqref{eq:6} are very similar.
For example, $\indt{X^D}$ is a mixed moving average process while $\indt {Y^D}$ is a mixed moving maxima process.  This and other
existing analogies motivated us to explore the structural relationship between sum-- and max--stable processes. 
In particular, while studying the max--stable processes, can we benefit from the known results for sum--stable processes? Is there any easy
way to `translate' results on S$\alpha$S processes to $\alpha$--Fr\'echet processes (or vice versa)?   We provide partial answers to these
questions in terms of the spectral representations of the sum-- and max--stable processes.  The following remark provides some important intuition.

\begin{Rem}\label{rem:1}
Any S$\alpha$S ($\alpha$--Fr\'echet resp.) process has many different representations in form of~\eqref{rep:integralRep} (\eqref{rep:extremalRep} resp.)..
All the representations for the same process can be related by linear (max--linear, to be defined in Section~\ref{sec:prelim}, resp.) isometries. Let us take S$\alpha$S processes for example. Namely, if $\indt {f\topp1}\subset\la(S_1,\mu_1)$ and $\indt {f\topp2}$ are two spectral functions for the same S$\alpha$S process $\indt X$, then 
\equh\label{eq:relating}
Uf_t\topp1\defe f_t\topp2\,,\forall t\in T
\eque 
defines a linear isometry between subspaces of $\la(S_1,\mu_1)$ and $\la(S_2,\mu_2)$ (generated by the spectral functions $\indt{f\topp 1}$ and $\indt{f\topp 2}$, see Section~\ref{sec:EPRS}). The fact, that $U$ is a linear isometry, follows from the characterization~\eqref{eq:3}, which implies 
\[
\bnn{\sum_{i=1}^na_if_{t_i}\topp1}_{\la(S_1,\mu_1)} = \bnn{\sum_{i=1}^na_iUf_{t_i}\topp1}_{\la(S_2,\mu_2)} \,,\forall a_i\in\mathbb R,t_i\in T,n\in\mathbb N\,.
\]
Similarly, because of~\eqref{eq:4}, any two spectral representations of the same $\alpha$--Fr\'echet process can be related through a {\it max--linear isometry}.

\end{Rem}
The fact that different spectral representations are related by linear (max--linear resp.) isometries implies that, roughly speaking,
all structural results and classifications of sum-- and max--stable processes based on spectral representations
must be invariant w.r.t.\ the linear (max--linear resp.) isometries.
Remark~\ref{rem:1} suggests that the isometries play an important role in the study of these processes. In fact, we will establish the following surprising result: the positive--linear and max--linear isometries are identical on the so--called \textit{extended positive ratio space} (Theorem~\ref{thm:1})..
This result enables us to associate S$\alpha$S processes and $\alpha$--Fr\'echet processes with the same spectral functions (Theorem~\ref{thm:association}). This association will serve as a tool to translate available structural results about S$\alpha$S processes to the domain of $\alpha$--Fr\'echet processes.
However, we will also observe that there are S$\alpha$S processes that cannot be associated to any $\alpha$--Fr\'echet processes (Theorem~\ref{thm:associable}). We provide a practical characterization of the {\it max--associable} S$\alpha$S processes $\indt X$ with stationary increments characterized by dissipative flow, indexed by $T= \mathbb R$ or $T=\mathbb Z$ (Proposition~\ref{prop:4}).
\smallskip\\
{\bf The paper is organized as follows.} In Section~\ref{sec:prelim}, we review some basic properties of stable and extremal stochastic integrals as well as the notions of positive--linear and max--linear isometries. In Section~\ref{sec:EPRS}, we identify the positive--linear and max--linear isometries on the \textit{extended positive ratio space}. In Section~\ref{sec:association}, we establish the association of S$\alpha$S and $\alpha$--Fr\'echet processes and provide examples of both max--associable and non max--associable S$\alpha$S processes. In Section~\ref{sec:classification}, we summarize some known classification results for S$\alpha$S and $\alpha$--Fr\'echet process, which can be related by the association method. In Section~\ref{sec:discussion}, we conclude with a short discussion on the comparison between~\citet{kabluchko08spectral} and our approach.
%%%%%%%%%%%%%%%%%%%%%%%%%%%%%%%%%%%%%%%%%%%%%%%%%%%%%%%%%%%%%%%%%%%%%%%%%%%%%%%%%%%%
%%%%%%%%%%%%%%%%%%%%%%%%%%%%%%%%%%%%%%%%%%%%%%%%%%%%%%%%%%%%%%%%%%%%%%%%%%%%%%%%%%%%
%%%%%%%%%%%%%%%%%%%%%%%%%%%%%%%%%%%%%%%%%%%%%%%%%%%%%%%%%%%%%%%%%%%%%%%%%%%%%%%%%%%%
%%%%%%%%%%%%%%%%%%%%%%%%%%%%%%%%%%%%%%%%%%%%%%%%%%%%%%%%%%%%%%%%%%%%%%%%%%%%%%%%%%%%
%%%%%%%%%%%%%%%%%%%%%%%%%%%%%%%%%%%%%%%%%%%%%%%%%%%%%%%%%%%%%%%%%%%%%%%%%%%%%%%%%%%%

\section{Preliminaries}\label{sec:prelim}
Here, we briefly review the properties of representations~\eqref{rep:integralRep} and~\eqref{rep:extremalRep} for S$\alpha$S and $\alpha$--Fr\'echet processes, respectively. For more details, see e.g.~\citet{samorodnitsky94stable} and~\citet{stoev06extremal}.\medskip

\noindent{\bf Symmetric $\alpha$--stable (S$\alpha$S) integrals}\smallskip

\itemnumber i (S$\alpha$S) The stable integral $Z\defe\int_Sf(s)d\mas(s)$ is well defined for all $f\in\la(S,\mu)$, $\alpha\in(0,2]$. It is an S$\alpha$S random variable with \textit{scale coefficient}
\[
\nn{Z}_\alpha = \left(\int_S\left|f(s)\right|^\alpha\mu(ds)\right)^{1/\alpha} = \nn f_{\la(S,\mu)}\,.
\]
\itemnumber {ii} (Independently scattered) For any $f,g\in\la(S,\mu), \alpha\in(0,2)$, $\int_Sfd\mas$ and $\int_Sgd\mas$ are independent, if and only if $fg = 0\,,\mu\ae$, i.e., $f$ and $g$ have disjoint supports.

\itemnumber {iii} (Linearity) For any $f,g\in\la(S,\mu)$, $a,b\in\mathbb R$, $\alpha\in(0,2]$, we have
\[
\Eintt_S(af(s) + bg(s))d\mas(s) = a\Eintt_Sf(s)d\mas(s) + b\Eintt_Sg(s)d\mas(s)\,,\mbox{a.e.}.
\]
\noindent{\bf Extremal integrals}\smallskip

\itemnumber i ($\alpha$--Fr\'echet) The extremal integral $Z\defe\eintt_Sf(s)d\mam(s)$ is well defined for all $f\in\lap(S,\mu)$, $\alpha\in(0,\infty)$. It is an $\alpha$--Fr\'echet random variable with \textit{scale coefficient}
\[
\nn{Z}_\alpha = \left(\int_Sf^\alpha(s)\mu(ds)\right)^{1/\alpha} = \nn f_{\lap(S,\mu)}\,.
\]
\itemnumber {ii} (Independently scattered) For any $f,g\in\lap(S,\mu), \alpha\in(0,\infty)$, $\Eintt_Sfd\mam$ and $\Eintt_Sgd\mam$, are independent, if and only if $fg = 0\,,\mu\ae$, i.e., $f$ and $g$ have disjoint supports.

\itemnumber {iii} (Max--linearity) For any $f,g\in\lap(S,\mu)$, $a,b>0$, $\alpha\in(0,\infty)$, we have
\[
\Eintt_S(af(s)\vee bg(s))d\mam(s) = a\Eintt_Sf(s)d\mam(s)\vee b\Eintt_Sg(s)d\mam(s)\,,\mbox{a.e.}.
\]
\newpara{Linear and max--linear isometries}
As we have mentioned in Remark~\ref{rem:1}, the linear isometries and max--linear isometries play important roles in relating two representations of a given S$\alpha$S or an $\alpha$--Fr\'echet process, respectively. The notion of a linear isometry is well known. We give next the definition of \textit{max--linear isometry}.
\begin{Def}[Max-linear isometry]
Let $\alpha>0$ and consider two measure spaces $(S_1,\mu_1)$ and $(S_2,\mu_2)$ with positive and $\sigma$-finite measures $\mu_1$ and $\mu_2$. The mapping $U:\lap(S_1,\mu_1)\to \lap(S_2,\mu_2)$, is said to be a max--linear isometry, if:\\
\itemnumber i For any $f_1,f_2\in \lap(S_1,\mu_1)$ and $\forall a_1,a_2\geq0$, $U(a_1f_1\vee a_2\,f_1) = a_1(Uf_1)\vee a_2(Uf_2), \mu_2\ae$ and\\
\itemnumber {ii} For any $f\in \lap(S_1,\mu_1)$, $\left\|Uf\right\|_{\lap(S_2,\mu_2)}=\left\|f\right\|_{\lap(S_1,\mu_1)}$.
\end{Def}
The linear (max--linear resp.) isometries may be naturally viewed as mappings between linear (max--linear resp.) spaces of functions. 
We say that a subset $\filF\subset\laps$ is a \textit{max--linear space} if for all $n\in\mathbb N, f_i\in\filF,a_i>0$, have $\bigvee_{i=1}^na_if_i\in\filF$ and if $\filF$ is closed w.r.t.\ the metric $\mrho$ defined by $\mrho(f,g) = \int_S|f^\alpha-g^\alpha|d\mu$\,.
A linear (max--linear resp.) isometry may be defined only on a small linear (max--linear resp.) subspace of $\la(S,\mu)$ ($\lap(S,\mu)$ resp.). It is important to understand what is the largest subspace of $\la(S,\mu)$ ($\lap(S,\mu)$ resp.), to which this isometry can be extended uniquely. The answer to this question involves the following notions of \textit{extended ratio spaces}.

\begin{Def}
[Extended ratio spaces] 
Let $F$ be a collection of functions in $\las$.

\itemnumber i The \textit{ratio $\sigma$-field} of $F$, written $\rho(F)\defe\sigma\left(\left\{f_1/f_2, f_1,f_2\in F\right\}\right)$, is defined as the $\sigma$-field generated by ratio of functions in $F$, where the ratios take values in the extended interval $[-\infty,\infty]$;\\
\itemnumber {ii} The \textit{extended ratio space} of $F$, written $\calR_e(F)$, is defined as the class of all functions in $\la(S,\mu)$ that has the form 
\equh\label{rep:ERS}
\calR_e(F) \defe \{rf:rf\in\la(S,\mu), r\in\rho(F), f\in F\}\,. 
\eque
Similarly, we define \textit{extended positive ratio space} of collection of functions $F\subset \laps$:
\equh\label{rep:EPRS}
\eratios F \defe \{rf:rf\in\lap(S,\mu), r\in\rho(F), r\geq 0, f\in F\}\,. 
\eque
\end{Def}
\noindent Note that $\calR_e(F)$ is closed w.r.t.\ linear combinations and the metric $(f,g)\mapsto \nn{f-g}_{\la(S,\mu)}^{1\wedge\alpha}$, and $\calR_{e,+}(F)$ is closed w.r.t.\ max--linear combinations and the metric $\rho_{\mu,\alpha}$. That is, $\calR_e(F)$ is a linear subspace of $\la(S,\mu)$ and $\calR_{e,+}(F)$ is a max--linear subspace of $\lap(S,\mu)$. 
The following result is due to~\citet{hardin81isometries} and~\citet{wang09structure}. 
\begin{Thm}\label{thm:factor}
Let $\filF$ be a linear (max--linear resp.) subspace of $\la(S_1,\mu_1)$ with $0<\alpha<2$. If $U$ is a linear (max--linear resp.) isometry from $\filF$ to ${U(\filF)}$, then $U$ can be uniquely extended to a linear (max--linear resp.) isometry $\widebar U: \calR_{e}(\filF)\to\calR_e(U(\filF))$ ($\widebar U:\calR_{e,+}(\filF)\to\calR_{e,+}(U(\filF))$ resp.), with the form
\equh\label{rep:factor}
\widebar U(rf) = \widebar T(r)U(f)\,,
\eque
for all $rf\in\calR_e(\filF)$ in~\eqref{rep:ERS} ($rf\in\eratiosf$ as in~\eqref{rep:EPRS} resp.).
Here $\widebar T$ is a mapping from $\la(S_1,\rho(\filF),\mu_1)$ to $\la(S_1,\rho(U(\filF)),\mu_2)$. $\widebar T$ is induced by a regular set isomorphism $T$ from $\rho(\filF)$ to $\rho(U(\filF))$.% $T$ is both max--linear and positive--linear.
\end{Thm}
\noindent For definition of \textit{regular set isomorphism}, see~\citet{lamperti58isometries},~\citet{hardin81isometries} or~\citet{wang09structure}. The following remark on Theorem~\ref{thm:factor}, particularly (iii), is crucial for the identification of two types of isometries.
\begin{Rem}\label{rem:2}
\itemnumber {i} $\widebar U$ is well defined in the sense that for any $r_if_i\in\eratiosf\,,i=1,2$ in~\eqref{rep:EPRS}, if $r_1f_1 = r_2f_2\,,\mu_1\ae$, then $\widebar U(r_1f_1) = \widebar U(r_2f_2)\,,\mu_2\ae$.\\
\itemnumber {ii} $T$ maps any two almost disjoint sets to almost disjoint sets.\\
\itemnumber {iii} Mapping $\widebar T$ is both max--linear and linear and maps nonnegative functions to nonnegative functions. This follows from the construction of $\widebar T$ via simple functions, and the fact that $\widebar T\ind_A = \ind_{T(A)}$ for measurable $A\subset S_1$. By (ii), for any simple functions $f = \sum_{i=1}^na_i\ind_{A_i}$ and $g = \sum_{j=1}^mb_j\ind_{B_j}$, where $A_i, B_j$ are mutually disjoint and $a_i,b_j\in\mathbb R$, we have
\[
\widebar T(f+g) = \widebar Tf + \widebar Tg \qmand \widebar T(f\vee g) = \widebar Tf\vee \widebar Tg\,.
\]

\itemnumber {iv}
When $\filF$ is a max--linear subspace and $U$ is a max--linear isometry, $\widebar U$ in~\eqref{rep:factor} is a linear isometry from $\calR_e(\filF)$ to $\calR_e(U(\filF))$. Indeed, by (iii), the max--linearity of $\widebar T$ implies linearity of $\widebar T$, and hence that of $\widebar U$. The isometry follows from the isometry for nonnegative functions and by (ii).
\end{Rem}
To make good use of (iii) in Remark~\ref{rem:2}, we introduce the notion of \textit{positive--linearity}.
We say a linear isometry $U$ is a \textit{positive--linear} isometry, if $U$ maps all nonnegative functions to nonnegative functions. Accordingly, 
we say that $\filF\subset\laps$ is a \textit{positive--linear space}, if it is closed w.r.t.\ $\mrho$ and if it is closed w.r.t positive--linear combinations, i.e., for all $n\in\mathbb N, f_i\in\filF,a_i>0$, we have $g\defe\sum_{i=1}^na_if_i\in\filF$. Note that the metric $(f,g)\mapsto\nn{f-g}_{\la(S,\mu)}^{1\wedge\alpha}$ restricted to $\lap(S,\mu)$ generates the same topology as the metric $\rho_{\mu,\alpha}$. Clearly, Theorem~\ref{thm:factor} holds if $\filF$ is a positive--linear (instead of a linear) subspace of $\laps$. In this case, $\widebar U$ is also positive--linear.
We conclude this section with the following refinement of statement (iii) in Remark~\ref{rem:2}.
\begin{Prop}\label{prop:extendedRatio2}
Let $U$ be as in Theorem~\ref{thm:factor}. If $\filF$ is a positive--linear subspace of $\lap(S_1,\mu_1)$, then the linear isometry $\widebar U$ in~\eqref{rep:factor} is also a max--linear isometry from $\eratiosf$ to $\eratios {U(\filF)}$. If $\filF$ is a max--linear subspace of $\la(S_1,\mu_1)$, then the max--linear isometry $\widebar U$ in~\eqref{rep:factor} is also a positive--linear isometry from $\calR_e(\filF)$ to $\calR_e(U(\filF))$.
\end{Prop}
\begin{proof}
Suppose $\filF$ is max--linear and $\widebar U$ is a max--linear isometry. We show $\widebar U$ is also positive--linear. The proof will be the same for the other case. First, (iv) of Remark~\ref{rem:2} implies $\widebar U$ is a linear isometry. Then, observe that $U$ maps nonnegative functions in $\filF$ to nonnegative functions in $U(\filF)$, and so does $\widebar T$. This shows that $\widebar U$ is a positive--linear isometry.
\end{proof}

%%%%%%%%%%%%%%%%%%%%%%%%%%%%%%%%%%%%%%%%%%%%%%%%%%%%%%%%%%%%%%%%%%%%%%%%%%%%%%%%%%%%
%%%%%%%%%%%%%%%%%%%%%%%%%%%%%%%%%%%%%%%%%%%%%%%%%%%%%%%%%%%%%%%%%%%%%%%%%%%%%%%%%%%%
%%%%%%%%%%%%%%%%%%%%%%%%%%%%%%%%%%%%%%%%%%%%%%%%%%%%%%%%%%%%%%%%%%%%%%%%%%%%%%%%%%%%
%%%%%%%%%%%%%%%%%%%%%%%%%%%%%%%%%%%%%%%%%%%%%%%%%%%%%%%%%%%%%%%%%%%%%%%%%%%%%%%%%%%%
%%%%%%%%%%%%%%%%%%%%%%%%%%%%%%%%%%%%%%%%%%%%%%%%%%%%%%%%%%%%%%%%%%%%%%%%%%%%%%%%%%%%
%%%%%%%%%%%%%%%%%%%%%%%%%%%%%%%%%%%%%%%%%%%%%%%%%%%%%%%%%%%%%%%%%%%%%%%%%%%%%%%%%%%%
%%%%%%%%%%%%%%%%%%%%%%%%%%%%%%%%%%%%%%%%%%%%%%%%%%%%%%%%%%%%%%%%%%%%%%%%%%%%%%%%%%%%
%%%%%%%%%%%%%%%%%%%%%%%%%%%%%%%%%%%%%%%%%%%%%%%%%%%%%%%%%%%%%%%%%%%%%%%%%%%%%%%%%%%%
%%%%%%%%%%%%%%%%%%%%%%%%%%%%%%%%%%%%%%%%%%%%%%%%%%%%%%%%%%%%%%%%%%%%%%%%%%%%%%%%%%%%

\section{Identification of Max--Linear and Positive--Linear Isometries}\label{sec:EPRS}
Here, we will first show that the max--linear and positive--linear isometries are identical on the extended positive ratio space. Then, we prove the following theorem, which is the main result of this section. It will be used to relate S$\alpha$S and $\alpha$--Fr\'echet processes in the next section. The results of Theorem~\ref{thm:factor} (see~\cite{hardin81isometries}) on linear isometries do not apply to the case
$\alpha = 2$. Thus, from now on, we shall assume 
\[
0<\alpha<2\,. 
\]
\begin{Thm}\label{thm:1}
Consider two arbitrary collections of functions $f_1\topp i,\dots,f_n\topp i\in\lap(S_i,\mu_i)\,, i=1,2,\ 0<\alpha<2$. Then,

\equh\label{eq:thm1sum}
\bnn{\sum_{j=1}^na_jf_j\topp 1}_{\la(S_1,\mu_1)} = \bnn{\sum_{j=1}^na_jf_j\topp 2}_{\la(S_2,\mu_2)}\,,\mbox{ for all } a_j\in\mathbb R\,,
\eque
if and only if
\equh\label{eq:thm1max}
\bnn{\bigvee_{j=1}^na_jf_j\topp 1}_{\lap(S_1,\mu_1)} = \bnn{\bigvee_{j=1}^na_jf_j\topp 2}_{\lap(S_2,\mu_2)}\,,\mbox{ for all } a_j\geq 0\,.
\eque
\end{Thm}
Before we can prove this theorem, we need an auxiliary result. 
We need to find a subspace of $\lap(S,\mu)$, which is closed w.r.t.\ the max--linear and positive--linear combinations. In the sequel, for any collection of functions $F\subset\laps$, we let
\equh\label{eq:filF}
\quad\filF_+ \defe \cpsspan\{F\}\qmand\filF_\vee \defe\cmsspan\{F\}
\eque
denote the smallest max--linear and positive--linear subspace of $\laps$ containing $F$, respectively. We call them the max--linear space and positive--linear space generated by $F$, respectively. (We also write $\filF \defe \csspan\{F\}$ as the smallest linear subspace of $\la(S,\mu)$ containing $F$.)
In general, we have $\filF_+\neq\filF_\vee$. This means both $\filF_+$ and $\filF_\vee$ are too small to be closed w.r.t.\ both `$\sum$' and `$\bigvee$' operators. However, we will show that these two subspaces generate the same extended positive ratio space, on which the two types of isometries are identical. The following fact is proved in the Appendix.
\begin{Prop}\label{prop:extendedRatio1}
Let $F$ be an arbitrary collection of functions in $\laps$. Then $\eratios {\filF_+} = \eratios {\filF_\vee}$.
\end{Prop}
\begin{proof}[Proof of Theorem~\ref{thm:1}]
Let $F\topp i\defe\{f_1\topp i,\dots,f_n\topp i\}\subset\lap(S_i,\mu_i),n\in\mathbb N$.
We prove the `only if' part. Suppose Relation~\eqref{eq:thm1sum} holds. We will show that Relation~\eqref{eq:thm1max} holds. 
There exists unique linear mapping $U$ from $\filF\topp1$ onto $\filF\topp2$, such that
\[
Uf_j\topp1 = f_j\topp 2\,,1\leq j\leq n.
\]
Note that since the functions $f_j\topp i$ are nonnegative, we have $U(\filF_+\topp 1) = \filF_+\topp 2$ and $U(\filF_\vee\topp 1) = \filF_\vee\topp 2$. 
Relation~\eqref{eq:thm1sum} implies that $U$ is a positive--linear isometry from $\filF_+\topp1$ to $U(\filF_+\topp1)$. Thus, Theorem~\ref{thm:factor} implies that the mapping 
\[
\widebar U:\calR_e(\filF_+\topp1)\to\calR_e(U(\filF_+\topp2))
\]
with form~\eqref{rep:factor} is a positive--linear isometry.
By Proposition~\ref{prop:extendedRatio1}, we have $\calR_e(\filF_+\topp i) = \calR_e(\filF_\vee\topp i),i=1,2$. Hence, $\widebar U$ is a positive--linear isometry from $\calR_e(\filF_\vee\topp1)$ to $\calR_e(U(\filF_\vee\topp1))$. By Proposition~\ref{prop:extendedRatio2}, $\widebar U$ is also a max--linear isometry from $\eratios{\filF_\vee\topp1}$ to $\eratios{U(\filF_\vee\topp1)}$, whence Relation~\eqref{eq:thm1max} holds. The proof of the `if' part is similar.
\end{proof}

To conclude this section, we will address the following question: for $f_1\topp 1,\dots,f_n\topp 1\in\la(S_1,\mu_1)$, does there always exist $f_1\topp2,\dots,f_n\topp2\in\lap(S_2,\mu_2)$ such that Relation~\eqref{eq:thm1sum} holds for any $a_j\in\mathbb R$? The answer is negative. 
\begin{Prop}\label{prop:associable}
Consider $f_j\topp 1\in\la(S_1,\mu_1), 1\leq j\leq n$. Then, there exist some $f_j\topp 2\in\lap(S_2,\mu_2)\,,1\leq j\leq n$ such that~\eqref{eq:thm1sum} holds, if and only if
\equh\label{cond:associable2}
f_i\topp1(s)f_j\topp1(s)\geq 0\,,\mu_1\ae \mbox{ for all } 1\leq i,j\leq n\,..
\eque
When \eqref{cond:associable2} is true, one can take $f_i\topp 2(s)\defe |f_i\topp1(s)|, 1\leq i\leq n$ and $(S_2,\mu_2) \equiv (S_1,\mu_1)$ for~\eqref{eq:thm1sum} to hold.
\end{Prop}
The proof is given in the Appendix. We will call~\eqref{cond:associable2} the \textit{associable condition}.
As a consequence, in the next section we will see that there are S$\alpha$S processes, which cannot be associated to any $\alpha$--Fr\'echet process.
%%%%%%%%%%%%%%%%%%%%%%%%%%%%%%%%%%%%%%%%%%%%%%%%%%%%%%%%%%%%%%%%%%%%%%%%%%%%%%%%%%%
%%%%%%%%%%%%%%%%%%%%%%%%%%%%%%%%%%%%%%%%%%%%%%%%%%%%%%%%%%%%%%%%%%%%%%%%%%%%%%%%%%%
%%%%%%%%%%%%%%%%%%%%%%%%%%%%%%%%%%%%%%%%%%%%%%%%%%%%%%%%%%%%%%%%%%%%%%%%%%%%%%%%%%%
%%%%%%%%%%%%%%%%%%%%%%%%%%%%%%%%%%%%%%%%%%%%%%%%%%%%%%%%%%%%%%%%%%%%%%%%%%%%%%%%%%%
%%%%%%%%%%%%%%%%%%%%%%%%%%%%%%%%%%%%%%%%%%%%%%%%%%%%%%%%%%%%%%%%%%%%%%%%%%%%%%%%%%%
%%%%%%%%%%%%%%%%%%%%%%%%%%%%%%%%%%%%%%%%%%%%%%%%%%%%%%%%%%%%%%%%%%%%%%%%%%%%%%%%%%%
\section{Association of Max-- and Sum--Stable Processes}\label{sec:association}
In this section, by essentially applying Proposition~\ref{prop:extendedRatio2} and~\ref{prop:extendedRatio1}, we associate an S$\alpha$S process to every $\alpha$-Fr\'echet process. The associated processes will be shown to have similar properties. However, we will also see that not all the S$\alpha$S processes can be associated to $\alpha$--Fr\'echet processes. We conclude with several examples.
First, inspired by the similarity in Representations~\eqref{rep:integralRep} and~\eqref{rep:extremalRep}, we introduce the following:
\begin{Def}[Associated spectral representations]\label{def:association}
We say that an S$\alpha$S process $\indt X$ and an $\alpha$--Fr\'echet process $\indt Y$ are \textit{associated}, if there exist $\indt f\subset\laps$ such that:
\[
\indt X \defd \bccbb{\int_{S}f_t d\mas}_{t\in T}\ \mbox{and}\quad \indt Y \defd \bccbb{\Eintt_Sf_td\mam}_{t\in T}\,.
\]
In this case, we say $\indt X$ and $\indt Y$ are associated by $\indt f$.
\end{Def}
\noindent The following result shows the consistency of Definition~\ref{def:association}, i.e., the notion of association is independent of the choice of spectral functions. 
\begin{Thm}\label{thm:association}
Suppose an S$\alpha$S process $\indt X$ and an $\alpha$--Fr\'echet process $\indt Y$ are associated by $\indt {f\topp1}\subset\lapdo$. Then, $\indt {f\topp2}\subset\lapim$ is a spectral representation of $\indt X$, if and only if it is a spectral representation of $\indt Y$. Namely,
\equh\label{eq:eintts1}
\Big\{\int_{S_1}f_t\topp1 d\mas\topp1\Big\}_{t\in T} \defd \Big\{\int_{S_2}f_t\topp2 d\mas\topp2\Big\}_{t\in T}
\Longleftrightarrow\Big\{\Eintt_{S_1}f_t\topp1 d\mam\topp1\Big\}_{t\in T} \defd \Big\{\Eintt_{S_2}f_t\topp2 d\mam\topp2\Big\}_{t\in T}\,,
\eque
where $\mas\topp i$ and $\mam\topp i$ are S$\alpha$S random measures and $\alpha$--Fr\'echet random sup--measures, respectively, on $S_i$ with control measure $\mu_i,i=1,2$.
\end{Thm}
\begin{proof}
First note that by~\eqref{eq:3}, the l.h.s.\ of~\eqref{eq:eintts1} is equivalent to:
\[
\bnn{\sum_{j=1}^na_jf_{t_j}\topp1}_{\la(S_1,\mu_1)} = \bnn{\sum_{j=1}^na_jf_{t_j}\topp2}_{\la(S_2,\mu_2)}\,\forall a_j\in\mathbb R\,, t_j\in T, 1\leq j\leq n, \forall n\in\mathbb N\,,
\]
which, by Theorem~\ref{thm:1}, is equivalent to:
\[
\bnn{\bigvee_{j=1}^na_jf_{t_j}\topp1}_{\lap(S_1,\mu_1)} = \bnn{\bigvee_{j=1}^na_jf_{t_j}\topp2}_{\lap(S_2,\mu_2)}\,\forall a_j\geq 0\,, t_j\in T, 1\leq j\leq n, \forall n\in\mathbb N\,.
\]
Since $t_1,\dots,t_n$ are arbitrary, the relation above, by~\eqref{eq:4}, is equivalent to the r.h.s.\ of~\eqref{eq:eintts1}.
\end{proof}
\noindent It follows from Theorem~\ref{thm:association} that, for the associated processes, it suffices to work on any spectral representations. The next result shows that the associated processes would be simultaneously stationary and self--similar. Here we assume $T = \mathbb R$ or $\mathbb Z$.
\begin{Coro}\label{coro:associatedStationary}
Suppose an S$\alpha$S process $\indt X$ and an $\alpha$--Fr\'echet process $\indt Y$ are associated. Then, 

\itemnumber i $\indt X$ is stationary if and only if $\indt Y$ is stationary..

\itemnumber {ii} $\indt X$ is self--similar with exponent $H$, if and only if $\indt Y$ is self--similar with exponent $H$.
\end{Coro}
\begin{proof}
Suppose $\indt X$ and $\indt Y$ are associated by $\indt f\subset\laps$. 
\noindent{\it (i)} For any $h\in T$, letting $g_t = f_{t+h}\,,\forall t\in T$, by stationarity of $\indt X$, we obtain $\indt g$ as another spectral representation. Namely,
\[
\bccbb{\int_Sf_td\mas}_{t\in T} \eqd \bccbb{\int_Sg_td\mas}_{t\in T}\,.
\]
By Theorem~\ref{thm:association}, the above statement is equivalent to 
\[
\bccbb{\Eintt_Sf_td\mam}_{t\in T} \eqd \bccbb{\Eintt_Sg_td\mam}_{t\in T}\,,
\]
which is equivalent to the fact that $\indt Y$ is stationary.

\noindent{\it (ii)} If $\indt X$ is self--similar with exponent $H$, then by definition, for any $a>0$, we have
\[%\equh\label{eq:xat1}
(X_{at_1},\dots,X_{at_n}) \defd \left(a^{H}X_{t_1},\dots,a^{H}X_{t_n}\right)\,.
\]%\eque
Set $g_t(s) = a^{H/\alpha}f_{t/a}(s)$. The same argument as in part {(i)} yields the result.
\end{proof}
It is obvious that not all $\alpha$--Fr\'echet processes can be associated to S$\alpha$S processes, as the latter requires $0<\alpha< 2$ while the former can take any $\alpha> 0$. On the other hand,
neither can all S$\alpha$S processes be associated to $\alpha$--Fr\'echet processes. 
This is because, not all S$\alpha$S processes have nonnegative spectral representations. For an S$\alpha$S process $\indt X$ with spectral representation $\indt f$ to have an associated $\alpha$--Fr\'echet process, a necessary and sufficient condition is that for any $t_1,\dots,t_n\in T$, $f_{t_1},\dots,f_{t_n}$ satisfy the associable condition~\eqref{cond:associable2}. We say such S$\alpha$S processes are \textit{max--associable}. Now, Proposition~\ref{prop:associable} becomes:
\begin{Thm}\label{thm:associable}
Any S$\alpha$S process $\indt X$ with representation~\eqref{rep:integralRep} is max--associable, if and only if for all $t_1,t_2\in T$, 
\equh\label{eq:ft1ft2}
f_{t_1}(s)f_{t_2}(s)\geq 0\,,\mu\ae\,.
\eque
\end{Thm}
\noindent Indeed, by Proposition~\ref{prop:associable} for any max--associable spectral representation $\indt f$, $\{|f_t|\}_{t\in T}$ is also a spectral representation for the same process. Clearly, if the spectral functions are nonnegative, then the S$\alpha$S processes are max--associable. We give two simple examples next. 
\begin{Rem}
All the examples in this section are well studied S$\alpha$S processes. We do not however list their properties in this paper. Many of the resulting associated $\alpha$--Fr\'echet processes are new, to the best of our knowledge. 
However, the association does not provide a complete picture of the probabilistic properties of these new $\alpha$--Fr\'echet processes. Their detailed studies present interesting problems for future research, which fall beyond the scope of this work.
\end{Rem}
\begin{Example}[Association of mixed fractional motions]
Consider the self--similar S$\alpha$S processes $\{X_t\}_{t\in\mathbb R_+}$ with the following representations
\equh\label{rep:mixedFractionalMotions}
\{X_t\}_{t\in \mathbb R_+} \eqd \bccbb{\int_E\int_0^\infty t^{H-\frac1\alpha}g\left(x,\frac ut\right)\mas(dx,du)}_{t\in \mathbb R_+}\,, H\in(0,\infty)\,,
\eque
where $(E,\calE,\nu)$ is a standard Lebesgue space, $\mas$ is an S$\alpha$S random measure on $X\times\mathbb R_+$ with control measure $m(dx,du) = \nu(dx)du$ and $g\in\la(E\times\mathbb R_+,m)$. Such processes are called \textit{mixed fractional motions} (see~\citet{burnecki98spectral}). When $g\geq 0$ a.e., the process $\{X_t\}_{t\in \mathbb R_+}$ is max--associable. The Corollary~\ref{coro:associatedStationary} implies the associated $\alpha$--Fr\'echet process is $H$--self--similar.
\end{Example}
\begin{Example}[Association of Chentzov S$\alpha$S Random Fields]
Recall that $\indtrn X$ is a Chentzov S$\alpha$S random field, if 
\[
\indtrn X  \equiv \{\mas(V_t)\}_{t\in \mathbb R^n} \eqd \bccbb{\int_S\ind_{V_t}(u)\mas(du)}_{t\in \mathbb R^n}\,.
\]
Here, $0<\alpha<2$, $(S,\mu)$ is a measure space and  $V_t, t\in  \mathbb R^n$ is a family of measurable sets such that $\mu(V_t)<\infty$ for all $t\in\mathbb R^n$ (see Ch.~8 in~\citet{samorodnitsky94stable}). Since $\ind_{V_t}(u)\geq 0$, all Chentzov S$\alpha$S random fields are max--associable.
\end{Example}
To conclude this section, we will show that there are S$\alpha$S processes that cannot be associated to any $\alpha$--Fr\'echet processes. In particular, recall that, the S$\alpha$S processes with stationary increments (zero at $t = 0$) characterized by dissipative flows were shown in~\citet{surgailis98mixing} to have representation
\equh\label{eq:XttR}
\{X_t\}_{t\in \mathbb R}\eqd \bccbb{\int_E\int_\mathbb R(G(x,t+u) - G(x,u))\mas(dx,du)}_{t\in \mathbb R}\,.
\eque
Here, $(E,\calE,\nu)$ is a standard Lebesgue space, $\mas,\alpha\in(0,2)$ is an S$\alpha$S random measure with control measure $m(dx,du) = \nu(dx)du$ and $G:E\times \mathbb R\to\mathbb R$ is a measurable function such that, for all $t\in \mathbb R$,
\[
G_t(x,u) = G(x,t+u)-G(x,u)\,, x\in E, u\in\mathbb R
\]
belongs to $L^\alpha(E\times\mathbb R,m)$. The process $\indtr X$ in~\eqref{eq:XttR} is called a \textit{mixed moving average with stationary increments}. Examples~\ref{example:3} and~\ref{example:4} show that not all such processes are max--associable. The following result provides a
partial characterization of the max--associable S$\alpha$S processes $\indt X$, which have the representation \eqref{eq:XttR}.
We shall suppose that $E$ is equipped with a metric $\rho$ and endow $E\times\mathbb R$ with the product topology.

\begin{Prop}\label{prop:4}
Consider an S$\alpha$S process $\indtr X$ with representation~\eqref{eq:XttR}. 
Suppose there exists a closed set $\calN\subset E\times\mathbb R$, such that $m(\calN) = 0$ and the function $G$ is continuous at all $(x,u)\in \calN^c\defe E\times\mathbb R\setminus \calN$, w.r.t.\ the product topology. Then, $\indtr X$ is max--associable, if and only if
\equh\label{eq:Gxu}
G(x,u)  = f(x)\ind_{A_x}(u) + c(x), \mmon \calN^c\,.
\eque
\end{Prop}
\begin{proof}
By Theorem~\ref{thm:associable}, $\indtr X$ is max--associable, if and only if for all $t_1,t_2\in\mathbb R$,
\equh\label{cond:associable3}
G_{t_1}(x,u)G_{t_2}(x,u) = (G(x,t_1+u)-G(x,u))(G(x,t_2+u)-G(x,u))\geq 0\,,m \ae\ \ \ (x,u)\in E\times\mathbb R\,.
\eque
First, we show the `if' part. Define $\widetilde G(x,u) \defe G(x,u)$ (given by~\eqref{eq:Gxu}) on $\calN^c$ and $\widetilde G(x,u) \defe f(x)\ind_{A_x}(u)+c(x)$ on $\calN$ (if $A_x$ and $c(x)$ is not defined, then set $\widetilde G(x,u) = 0$). Set $\widetilde G_t(x,u) = \widetilde G(x,u+t) - \widetilde G(x,u)$. 
Note that $\widetilde G_t(x,u)$ is another spectral representation of $\indtr X$ and for all $(x,u)$, $\bccbb{\ind_{A_x}(u+t)-\ind_{A_x}(u)}$ can take at most 2 values, one of which is $0$. This observation implies~\eqref{cond:associable3} with $G_t(x,u)$ replaced by $\widetilde G_t(x,u)$, whence $\indtr X$ is max--associable.

Next, we prove the `only if' part. We show that~\eqref{cond:associable3} is violated, if $G(x,u)$ takes more than 2 different values on $(\{x\}\times\mathbb R)\cap\calN^c$ for some $x\in X$. Suppose there exist $\exists x\in E, u_i\in \mathbb R$ such that $(x,u_i)\in\calN^c$ and $g_{xi}\defe G(x,u_i)$ are mutually different, for $i = 1,2,3$. Indeed, without loss of generality we may suppose that $g_{x1}<g_{x2}<g_{x3}$. Then, by the continuity of $G$, there exists $\epsilon>0$ such that $B_i\defe B(x,\epsilon) \times (u_i-\epsilon,u_i+\epsilon)\,, i = 1,2,3$ are disjoint sets with $B(x,\epsilon) \defe\{y\in E:\rho(x,y)<\epsilon\}$, $\rho$ is the metric on $E$ and 
\equh\label{eq:B1N}
\sup_{B_1\cap\calN^c}G(x,u)<\inf_{B_2\cap\calN^c}G(x,u)\leq \sup_{B_2\cap\calN^c}G(x,u)<\inf_{B_3\cap\calN^c}G(x,u)\,.
\eque
Put $t_1 = u_1-u_2$ and $t_2 = u_3-u_2$. Inequality~\eqref{eq:B1N} implies that $G_{t_1}(x,u)G_{t_2}(x,u)< 0$ on $B_2\cap\calN^c$. This, in view of Theorem~\ref{thm:associable}, contradicts the max--associability. Therefore, for all $x\in E$, $G(x,u)$ can take at most two values on $\calN^c$, which implies~\eqref{eq:Gxu}.
\end{proof}
\noindent We give two classes of S$\alpha$S processes, which cannot be associated to any $\alpha$--Fr\'echet processes, according to Proposition~\ref{prop:4}.

\begin{Example}[Non--associability of linear fractional stable motions]\label{example:3}
The {\it linear fractional stable motions} (see Ch.~7.4 in~\citet{samorodnitsky94stable}) have the following spectral representations:
\[
\indtr X \eqd \bccbb{\int_{\mathbb R}\bccbb{a\left((t+u)_+^{H-1/\alpha} - u_+^{H-1/\alpha}\right) + b\left((t+u)_-^{H-1/\alpha} - u_-^{H-1/\alpha}\right)}\mas(du)}_{t\in \mathbb R}\,.
\]
Here $H\in(0,1)$, $\alpha\in(0,2)$, $H\neq 1/\alpha$, $a,b\in\mathbb R$ and $|a|+|b| > 0$. By Proposition~\ref{prop:4}, these processes are not max--associable. 
\end{Example}
\begin{Example}[Non--associability of Telecom processes]\label{example:4}
The Telecom process offers an extension of fractional Brownian motion consistent with heavy--tailed fluctuations. It is a large scale limit of renewal reward processes and it can be obtained by choosing the distribution of the rewards accordingly (see~\citet{levy00renewal} and~\citet{pipiras04slow}). A Telecom process $\indtr X$ has the following representation
\[
\indtr X \eqd \bccbb{\int_{\mathbb R}\int_{\mathbb R}e^{(H-1)/\alpha}\left(F(e^s(t+u)) - F(e^su)\right)\mas(ds,du)}_{t\in \mathbb R}\,,
\]
where $1<\alpha<2$, $1/\alpha<H<1$, $F(z) = (z\wedge 0 + 1)_+\,, z\in\mathbb R$ and the S$\alpha$S random measure $\mas$ is with control measure $m_\alpha(ds,du) = dsdu$. By Proposition~\ref{prop:4}, the Telecom process is not max--associable.
\end{Example}
\begin{Rem}
It is important that the index $T$ in Proposition~\ref{prop:4} is the entire real line $\mathbb R$. 
Indeed, in both Example~\ref{example:3} and~\ref{example:4}, when the time index is restricted to the half--line $T = \mathbb R_+$ (or $T = \mathbb R_-$), the processes $\indt X$ satisfy condition~\eqref{eq:ft1ft2} and are therefore max--associable.
\end{Rem}

%%%%%%%%%%%%%%%%%%%%%%%%%%%%%%%%%%%%%%%%%%%%%%%%%%%%%%
%%%%%%%%%%%%%%%%%%%%%%%%%%%%%%%%%%%%%%%%%%%%%%%%%%%%%%
%%%%%%%%%%%%%%%%%%%%%%%%%%%%%%%%%%%%%%%%%%%%%%%%%%%%%%
%%%%%%%%%%%%%%%%%%%%%%%%%%%%%%%%%%%%%%%%%%%%%%%%%%%%%%
%%%%%%%%%%%%%%%%%%%%%%%%%%%%%%%%%%%%%%%%%%%%%%%%%%%%%%
%%%%%%%%%%%%%%%%%%%%%%%%%%%%%%%%%%%%%%%%%%%%%%%%%%%%%%
%%%%%%%%%%%%%%%%%%%%%%%%%%%%%%%%%%%%%%%%%%%%%%%%%%%%%%
\section{Association of Classifications}\label{sec:classification}
We can also apply the association technique to relate various classification results for S$\alpha$S and $\alpha$--Fr\'echet processes. Note that, many classifications of S$\alpha$S ($\alpha$--Fr\'echet as well) processes are induced by suitable decompositions of the measure space $(S,\mu)$. The following theorem provides an essential tool for translating any classification results for S$\alpha$S to $\alpha$--Fr\'echet processes, and vice versa.
\begin{Thm}\label{thm:2}
Suppose an $\alpha$--Fr\'echet process $\indt X$ and an S$\alpha$S process $\indt Y$ are associated by two spectral representations $\indt {f\topp i}\subset\lap(S_i,\mu_i)$ for $i = 1,2$. That is,
\[
\indt X \eqd \left\{\Eintt_{S_i}f_t\topp i\d\mam\topp i\right\}_{t\in T} \qmand
\indt Y \eqd \left\{\int_{S_i}f_t\topp i\d\mam\topp i\right\}_{t\in T}\,,i = 1,2\,.
\]
Then, for any measurable subsets $A_i\subset S_i,i=1,2$, we have
\[
\bccbb{\Eintt_{A_1}f_t\topp 1d\mam\topp1}_{t\in T} \eqd \bccbb{\Eintt_{A_2}f_t\topp 2d\mam\topp2}_{t\in T}\Longleftrightarrow
\bccbb{\int_{A_1}f_t\topp 1d\mas\topp1}_{t\in T} \eqd \bccbb{\int_{A_2}f_t\topp 2d\mas\topp2}_{t\in T}\,.
\]
\end{Thm}
\noindent The proof follows from Theorem~\ref{thm:1} by restricting the measures onto the sets $A_i, i = 1,2$.\smallskip

For an S$\alpha$S process $\indt X$ with spectral functions $\indt f\subset\las$, a decomposition typically takes the form
\equh\label{eq:XinX}
\indt X\eqd \bccbb{\sum_{j=1}^nX_t\topp j}_{t\in T}\,,
\eque
where $X_t\topp j = \int_{A\topp j}f_t(s)d\mas(s)$ for all $t\in T$ and $A\topp j,1\leq j\leq n$ are disjoint subsets of $S = \bigcup_{j=1}^nA\topp j$. The components $\indt{X\topp j}\,,1\leq j\leq n$ are independent S$\alpha$S processes. 
When $\indt X$ is max--associable, Theorem~\ref{thm:2} enables us to define the \textit{associated decomposition}, for the $\alpha$--Fr\'echet process $\indt Y$ associated with $\indt X$. Namely, we have
\[
\indt Y\eqd \bccbb{\bigvee_{j=1}^nY_t\topp j}_{t\in T}\,,
\]
where $Y\topp j_t = \Eintt_{A^{(j)}}|f_t(s)|d\mam(s)$ for all $t\in T$. Similarly, we can define the associated decomposition for S$\alpha$S process based on the decomposition of the associated $\alpha$--Fr\'echet processes.

We list below several known classification results for S$\alpha$S and $\alpha$--Fr\'echet processes. 
The decompositions below were obtained independently for sum-- and max--stable processes, without the use of association (see~\citet{hardin81isometries},~\cite{hardin82spectral},~\cite{rosinski95structure},~\cite{samorodnitsky05null} for S$\alpha$S processes and~\cite{dehaan84spectral} and~\cite{wang09structure} for $\alpha$--Fr\'echet processes). Theorem~\ref{thm:2} provides a simple way to relate these decompositions as well as to translate any (new) classification results from the sum--stable world to the max--stable world, and vice versa.

For the sake of simplicity, we present the results only for $\alpha$--Fr\'echet processes. In order to obtain the corresponding results for S$\alpha$S processes, it suffices to replace all the $\vee$ operators by $+$ and replace all the extremal integrals in form of $\ \Eintt_Sfd\mam$ by S$\alpha$S integrals $\int_Sfd\mas$ (then $f$ can be in $\la(S,\mu)$ instead of $\lap(S,\mu)$).
Consider any measurable stationary $\alpha$--Fr\'echet process $\indt Y$ with spectral representation~\eqref{rep:extremalRep} and assume that $T = \mathbb Z$ with $\lambda(dt)$ being the counting measure or $T = \mathbb R$ with $\lambda(dt)$ being the Lebesgue measure. We have: \medskip

\itemnumber i{\bf Conservative--dissipative decomposition:}
\[
\indt Y \eqd \{Y^C_t \vee Y^D_t\}_{t\in T}\,.
\]
Here $Y^C_t = \Eintt_C f_t(s)\mam(ds)$ and $Y^D_t = \Eintt_D f_t(s)\mam(ds)$ for all $t\in T$, with $C$ and $D$ defined by
\equh\label{decomp:CD}
C \defe \bccbb{s:\int_Tf_t(s)^\alpha \lambda(dt) = \infty}\qmand 
D \defe S\setminus C\,.%\bccbb{s:\int_T\left|f_t\topp1(s)\right|^\alpha dt < \infty}\,.
\eque
The sets $C$ and $D$ correspond to the {\it Hopf decomposition} $S = C\cup D$ of the non--singular flow associated with $\indt Y$ (see e.g.~\citet{rosinski95structure} and~\citet{wang09structure}). Thus, $\indt{Y^C}$ and $\indt{Y^D}$ are referred to as the {\it conservative} and {\it dissipative components} of $\indt Y$, respectively. Obviously, if $C$ ($D$ resp.) has zero measure, then $\indt {Y^C}$ ($\indt {Y^D}$ resp.) is trivial.\medskip\\
\itemnumber {ii}{\bf Positive--null decomposition:}
\equh\label{decomp:PN}
\indt Y = \{Y^{P}_t \vee Y^{N}_t\}_{t\in T}\,.
\eque
Here $Y^{P}_t = {\Eintt_{P}f_t(s)\mam(ds)}$ and ${Y^{N}_t} = \Eintt_{N}f_t(s)\mam(ds)\,,\forall t\in T$, with $P$ and $N$ defined as follows. Let $\calW$ be the class of functions $w:T\to\mathbb R_+$ such that $w$ is nondecreasing on $T\cap(-\infty,0]$, nonincreasing on $T\cap[0,\infty)$ and $\int_{T\cap(-\infty,0]}w(t)\lambda(dt) = \int_{T\cap[0,\infty)}w(t)\lambda(dt) = \infty\,.$ Here $\lambda$ is the counting measure if $T = \mathbb Z$ and the Lebesgue measure if $T = \mathbb R$. Then, $S$ can be decomposed into two parts, $S = {P}\cup {N}$, where
\eqnhn
{P} & \defe & \bccbb{s\in S:\int_T w(t)f_t(s)^\alpha\lambda(dt) = \infty\,,\mbox{ for all } w\in\calW}\label{decomp:nullPositive1}\qmand N \defe S\setminus P\,.
\eqnen
Note that $\mu(D\setminus N) = 0$ and $\mu(P\setminus C) = 0$. 
This implies that $\indt{Y^D}$ has no positive component, and one can combine~\eqref{decomp:CD} and~\eqref{decomp:PN} as follows:
\[
\indt Y \eqd \{Y_t^P \vee Y_t^{C,N}\vee Y_t^D\}_{t\in T}\,,
\]
where $\indt{Y^{C,N}} \equiv\{\ \Eintt_{C\cap N}f_t\mam\}_{t\in T}$ and $\indt{Y^P},\indt{Y^{C,N}}$ and $\indt{Y^D}$ are independent.
The components $Y^{P}$ and $Y^{\rm D}$ have relatively clear structures: $\indt{Y^D}$ is a mixed moving maxima (see~\citet{wang09structure}). Similarly, in the S$\alpha$S setting, the dissipative component $X^D = \indt{X^D}$ is a mixed moving average (see~\cite{rosinski95structure}). For a description of $\indt {X^P}$ ($\indt{Y^P}$ resp.), see e.g.~\citet{samorodnitsky05null}..
At the same time, the characterization for the component $\indt{Y^{C,N}}$ (or $\indt{X^{C,N}}$) is an open problem.

\begin{Rem}
We do not exhaust here all the structural classification results for sum--stable processes. ~\cite{pipiras02structure}, for example, provide a more detailed decomposition for S$\alpha$S processes with representation~\eqref{eq:XttR}. By using association, one can automatically obtain corresponding decompositions for the associated $\alpha$--Fr\'echet processes.
\end{Rem}

\section{Discussion}\label{sec:discussion}
Recently,~\citet{kabluchko08spectral} introduced a similar notion of association. We became aware of his result toward the end of our work. The two approaches are technically different. Kabluchko's approach utilizes \textit{spectral measures}, while ours is based on the structure of max--linear and linear isometries. These two approaches lead to the equivalent notions of association (see Lemma~2 in~\citet{kabluchko08spectral}). As a consequence, our Corollary~\ref{coro:associatedStationary} can also be obtained following his approach. 
On the other hand, our approach leads to a more direct proof of the following, which is Lemma 3 in~\citet{kabluchko08spectral}.
\begin{Lem}
Let $\indt X$ be an S$\alpha$S process and $\indt Y$ be an $\alpha$--Fr\'echet process. Suppose $\indt X$ and $\indt Y$ are associated by $\indt f\subset\lap(S,\mu)$. Then for any $\inftydots t\in T$, as $n\to\infty$, $X_{t_n}$ converges in probability to $X_t$, if and only if $Y_{t_n}$ converges in probability to $Y_t$. 
\end{Lem}
\begin{proof}
By Proposition 3.5.1 in~\citet{samorodnitsky94stable}, $X_{t_n}\stackrel{p}{\to} X_t$ as $n\to\infty$, if and only if $\left\|f_{t_n}-f_t\right\|_{\lap(S,\mu)}^\alpha \to 0$ as $n\to\infty$. This is equivalent to, by Theorem 2.1 and Lemma 2.3 in~\cite{stoev06extremal}, $Y_{t_n} \stackrel{p}{\to} Y_t$ as $n\to\infty$.
\end{proof}
\citet{kabluchko08spectral} also proved (Theorem 9 therein) that an $\alpha$--Fr\'echet process is mixing (ergodic resp.) if and only if the associated S$\alpha$S process is mixing (ergodic resp.). The proofs of these results, however, are not
simple consequences of the notion of association. 
By association one can easily obtain new classes of $\alpha$--Fr\'echet processes. 
The probabilistic properties of the new processes (e.g.\ the associated $\alpha$--Fr\'echet processes in examples in Section~\ref{sec:association}), however, do not automatically follow `by association' and are yet to be investigated.
\section*{Acknowledgments}
The authors were partially supported by NSF grant DMS--0806094 at the University of Michigan.

\appendix
\section{Proofs of Auxiliary Results}\label{sec:proofs}
We first need the following lemma.
\begin{Lem}\label{lem:equalSigma}
If $F \subset\laps$, then 

\itemnumber {i} $\rho(F) = \rho(\cpsspan(F)) = \rho(\cmsspan(F))$, and

\itemnumber {ii} for any $f\topp 1\in\cpsspan(F)$ and $f\topp2\in\cmsspan(F)$, $f\topp1/f\topp2\in\rho(F)$.
\end{Lem}
\begin{proof}\itemnumber{i}
First, for any $f_i,g_i\in F,a_i\geq 0,b_i\geq 0,i\in\mathbb N$, we have
\[
\bccbb{\frac{\bigvee_{i\in\mathbb N}a_if_i}{\bigvee_{j\in\mathbb N}b_jg_j}\leq x} = \bigcap_{i\in\mathbb N}\bccbb{\frac{a_if_i}{\bigvee_{j\in\mathbb N} b_jg_j}\leq x} = 
\bigcap_{i\in\mathbb N}\bigcap_{k\in \mathbb N}\bigcup_{j\in\mathbb N}\bccbb{\frac{a_if_i}{b_jg_j}< x+\frac1k} \,,
\]
hence $\rho(\cmsspan(F)) \subset\rho(\cpsspan(F))$. 

To show $\rho(\cpsspan(F))\subset\rho(\cmsspan(F))$, we shall first prove that $\rho(\psspan(F)) \subset\rho(\cmsspan(F))$, where $\psspan(F)$ involves only finite positive linear combinations. For all $f_1,f_2,g_1\in F, a_1,b_1,b_2\geq 0$, we have
\[
\bccbb{\frac{a_1f_1+a_2f_2}{b_1g_1}\leq x} = \bigcup_{q_j\in\mathbb Q}\Big(\bccbb{\frac{a_1f_1}{b_1g_1}\leq q_j} \cap \bccbb{\frac{a_2f_2}{b_1g_1}\leq x-q_j}\Big)\,,
\]
This shows that $(a_1f_1+a_2f_2)/b_1g_1$ is $\rho(\cmsspan(F))$ measurable. By using the fact that $F$ contains only nonnegative functions and since $\left\{\frac{b_1g_1}{a_1f_1+a_2f_2}\leq x\right\} = \left\{\frac{a_1f_1+a_2f_2}{b_1g_1}\geq \frac1x\right\}$, for $x>0$, we similarly obtain that $(a_1f_1+a_2f_2)/(b_1g_1+b_2g_2)$ is $\rho(\cmsspan(F))$ measurable. Similarly arguments can be used to show that $(\sumin a_if_i)/(\sumin b_ig_i)$ is $\rho(\cmsspan(F))$ measurable for all $a_i,b_i\geq 0, f_i, g_i\in F, 1\leq i\leq n$. 

We have thus shown that $\rho(\psspan(F))\subset\rho(\cmsspan(F))$. If now $f,g\in\cpsspan(F)$, then there exist two sequences $f_n,g_n\in\psspan(F)$, such that $f_n\to f$ and $g_n\to g$ a.e.. Thus, $h_n\defe f_n/g_n \to h\defe f/g$ as $n\to\infty$, a.e.. Since $h_n$ are $\rho(\psspan(F))$ measurable for all $n\in\mathbb N$, so is $h$.
Hence $\rho(\cpsspan(F)) = \rho(\psspan(F))\subset\rho(\cmsspan(F))$.

\itemnumber {ii} By the previous argument, it is enough to focus on finite linear and max--linear combinations. Suppose $f\topp1 = \sumin a_if_i$ and $f\topp2 = \bveejp b_jg_j$ for some $f_i,g_j\in F, a_i,b_j\geq 0, 1\leq i\leq n, 1\leq j\leq p$. Then, for all $x>0$,
\[%\equh\label{eq:aifi}
\bccbb{\frac{\sumin a_if_i}{\bveejp b_jg_j}<x} = \bigcup_{j=1}^p\bccbb{\sumin a_i\frac{f_i}{g_j}<x{b_j}} \in\rho(F)\,.
\]%\eque
It follows that $f\topp1/f\topp2\in\rho(F)$.
\end{proof}
\begin{proof}[Proof of Proposition~\ref{prop:extendedRatio1}]
First we show $\calR_{e,+}(\filF_\vee)\supset \calR_{e,+}(\filF_+)$, where $\filF_\vee$ and $\filF_+$ are defined in~\eqref{eq:filF}. By~\eqref{rep:EPRS}, it suffices to show that, for any $r_2\in\rho(\filF_+), f\topp2\in\filF_+$, there exist $r_1\in\rho(\filF_\vee)$ and $f\topp1\in\filF_\vee$, such that 
\equh\label{eq:r1f1}
r_1f\topp1 = r_2f\topp2\,.
\eque
To obtain~\eqref{eq:r1f1}, we need the concept of {\it full support}. We say a function $g$ has full support in $F$ (an arbitrary collection of functions defined on $(S,\mu)$), if $g\in F$ and for all $f\in F$, $\mu(\supp(g)\setminus\supp(f)) = 0$. Here $\supp(f) \defe \{s\in S:f(s)\neq 0\}$. By Lemma 3.2 in~\citet{wang09structure}, there exists function $f\topp1\in\filF_\vee$, which has full support in $\filF_\vee$.  One can show that this function has also full support in ${\cal F}_+$. 
Indeed, let $g\in\filF_+$ be arbitrary. Then, there exist $g_n = \sum_{i=1}^{k_n}a_{ni}g_{ni}, a_{ni}\geq 0$ and $g_{ni}\in F\subset\filF_\vee$ such that $g_n\stackrel{\mu}{\longrightarrow}g$ as $n\to\infty$. Note that $\mu(\supp(g_n)\setminus\supp(f)) = 0$ for all $n$. Thus, for all $\epsilon>0$, we have $\mu(|g_n-g|>\epsilon) \geq \mu(\{|g|>\epsilon\}\setminus \supp(f))$. Since $\mu(|g_n-g|>\epsilon)\to0$ as $n\to\infty$, it follows that $\mu(\{|g|>\epsilon\}\setminus\supp(f)) = 0$ for all $\epsilon>0$, i.e., $\mu(\supp(g)\setminus\supp(f)) = 0$.  We have thus shown that $f$ has full support in ${\cal F}_+$.

Now, set $r_1\defe r_2\left(f\topp2/f\topp1\right)$, we have~\eqref{eq:r1f1}. (Note that $f\topp2 = 0\,,\mu\ae$ on $S\setminus\supp(f\topp1)$. By setting $0/0 = 0$, $f\topp2/f\topp1$ is well defined.) Lemma~\ref{lem:equalSigma} (ii) implies that $f\topp2/f\topp1\in\rho(F)$, whence $r_1\in\rho(F) = \rho(\filF_+)$. We have thus shown $\calR_{e,+}(\filF_\vee)\supset\calR_{e,+}(\filF_+)$. In a similar way one can show $\calR_{e,+}(\filF_\vee)\subset\calR_{e,\vee}(\filF_+)$.
\end{proof}
\begin{proof}[Proof of Proposition~\ref{prop:associable}]
First, suppose~\eqref{cond:associable2} does not hold but~\eqref{eq:thm1sum} holds. Then, without loss of generality, we can assume that there exists $S_0\topp1\subset S_1$ such that $f_1\topp1(s)>0, f_2\topp2(s)<0$ for all $s\in S_0\topp1$ and $\mu(S_0\topp1)>0$. It follows from~\eqref{eq:thm1sum} that there exists a linear isometry $U$ such that, by Theorem~\ref{thm:factor}, $Uf_i\topp1 = f_i\topp2 = \widebar T(r_i)U(f)$, with certain $f$ and $r_i = f_i\topp1/f$, for $i = 1,2$. 
In particular, $f$ can be taken with full support.
Note that $\sign(r_1)\neq\sign(r_2)$ on $S_0\topp1$. It follows that $f_1\topp2$ and $f_2\topp2$ have different signs on a set of positive measure (indeed, this set is the image of the $S_0\topp1$ under the regular set isomorphism $T$). This contradicts the fact that $f_1\topp2$ and $f_2\topp2$ are both nonnegative on $S_2$.

On the other hand, suppose~\eqref{cond:associable2} is true. Define $Uf_i\topp1 \defe |f_i\topp1|$. It follows from~\eqref{cond:associable2} that $U$ can be extended to a positive--linear isometry from $\la(S_1,\mu_1)$ to $\lap(S_2,\mu_2)$, which implies~\eqref{eq:thm1sum}.
\end{proof}
\bibliographystyle{elsarticle-harv}

%\bibliography{../include/maxstable}

\end{document}